\documentclass[preprint]{article}%
\usepackage{graphicx}
\usepackage{amsmath,amsxtra,amssymb,latexsym, amscd,amsthm}
\usepackage[mathscr]{eucal}
\newtheorem{thm}{Theorem}[section]

\newtheorem{lem}{Lemma}[section]
\newtheorem{prop}{Proposition}[section]
\theoremstyle{definition}

\theoremstyle{remark}
\newtheorem{rem}{Remark}[section]
\numberwithin{equation}{section}
\def\ind{{\rm 1\hspace{-0.90ex}1}}

\begin{document}

\title{Non-uniform Berry-Esseen bounds via Malliavin-Stein method}
\author{Nguyen Tien Dung\thanks{Department of Mathematics, VNU University of Science, Vietnam National University, Hanoi, 334 Nguyen
Trai, Thanh Xuan, Hanoi, 084 Vietnam. Email: dung@hus.edu.vn}\and Le Vi$^\ast$\and Pham Thi Phuong Thuy\thanks{The faculty of Basic Sciences, Vietnam Air Defence and Air Force Academy, Son Tay, Ha Noi, 084 Vietnam.}}

\date{\today}          

\maketitle
\begin{abstract} In this paper, we establish non-uniform Berry-Esseen bounds by means of the Malliavin-Stein method. Applications to the multiple Wiener-It\^o integrals and the exponential functionals of Brownian motion are given to illustrate the theory.
\end{abstract}
\section{Introduction}
The Malliavin-Stein method was introduced by Nourdin and Peccati in 2009. Their seminal paper \cite{NourdinPe2009} combines Stein's method and Malliavin calculus to obtain quantitative central limit theorems for the functionals of a Gaussian process. Since then, many important achievements have been obtained by various authors. For a long list of papers using the Malliavin-Stein method, the reader can consult the constantly updated website
$$\text{https://sites.google.com/site/malliavinstein/home.}$$
One of main results obtained in \cite{NourdinPe2009} is the following uniform Berry-Esseen bound: Let $Z\sim N(0,1)$ and let $F\in \mathbb{D}^{1,2}$ with zero mean. If the law of $F$ is absolutely continuous with respect to the Lebesgue measure, then we have
\begin{equation}\label{whjo}
\sup\limits_{z\in \mathbb{R}}|P(F\leq z)-P(Z\leq z)|\leq \sqrt{E|1-\langle DF,-DL^{-1}F\rangle_{\mathfrak{H}}|^2}.
\end{equation}
The space $\mathbb{D}^{1,2},$ the operators $D$ and $L^{-1}$ will be defined in Section \ref{uufk} below. We recall that the uniform Berry-Esseen bounds are very useful in statistical applications. However,  such bounds only work well for values of $z$ near the center of the distribution. For $z$ sufficiently large, the difference $P(F\leq z)-P(Z\leq z)$ becomes so close to zero and the uniform bounds are too crude to be of any use. 

In this paper, our aim is to obtain the non-uniform Berry-Esseen bounds standing for (\ref{whjo}). More specifically, under the same assumption as in \cite{NourdinPe2009}, we obtain in Theorem \ref{kfoi} below the following bound
\begin{equation}\label{h4jo}
|P(F\leq z)-P(Z\leq z)|\leq \sqrt{E|1-\langle DF,-DL^{-1}F\rangle_{\mathfrak{H}}|^2}(\sqrt{P(|F|>|z|/2)}+2e^{-z^2/4})\,\,\,\forall\,z\in \mathbb{R}.
\end{equation}
In the classical central limit theorem for independent random variables, the rate of convergence with respect to $z$ is of order $O(1+|z|^3).$ Here, by using Markov's inequality, our bound (\ref{h4jo}) also gives the same this rate whenever $E|F|^6$ is finite. Furthermore, we can employ the theory of concentration inequalities to evaluate $P(|F|>|z|/2)$ and, in many situations, we are able to obtain a much faster rate. For example, when applied to $F=I_q(f)$ the multiple Wiener-It\^o integral of order $q\geq 2,$ the bound (\ref{h4jo}) gives us an exponential rate that reads
$$|P(F\leq z)-P(Z\leq z)|\leq \sqrt{\frac{q-1}{3q}(E|F|^4-3)}(c_qe^{-\frac{z^{2/q}}{2^{2+2/q}}}+2e^{-\frac{z^2}{4}})\,\,\,\forall\,z\in \mathbb{R},$$
where $c_q$ is a positive constant.

The rest of this article is organized as follows. In Section \ref{uufk}, we briefly recall some of the relevant elements of the Malliavin calculus. Our main results are then formulated and proved in Section \ref{uloq}. In Section \ref{9jko}, we apply our results to derive non-uniform Berry-Esseen bounds for the multiple Wiener-It\^o integrals and the exponential functionals of Brownian motion.
\section{Malliavin Calculus}\label{uufk}
Let us recall some elements of Malliavin calculus that we need in order to perform our proofs (for more details see \cite{nualartm2}). Suppose that $\mathfrak{H}$ is a real separable Hilbert space with scalar product denoted by $\langle.,.\rangle_\mathfrak{H}.$ We denote by $W = \{W(h) : h \in \mathfrak{H}\}$ an isonormal Gaussian process defined in a complete probability space $(\Omega,\mathcal{F},P),$ $\mathcal{F}$ is the $\sigma$-field generated by $W.$ Let $\mathcal{S}$ be the set of all smooth cylindrical random variables of the form
\begin{equation}\label{ro}
F=f(W(h_1),...,W(h_n)),
\end{equation}
where $n\in \mathbb{N}, f\in C_b^\infty(\mathbb{R}^n)$ the set of bounded and infinitely differentiable functions with bounded partial derivatives, $h_1,...,h_n\in \mathfrak{H}.$ If $F$ has the form (\ref{ro}), we define its Malliavin derivative with respect to $W$ as the element of $L^2(\Omega,\mathfrak{H})$ given by
$$DF=\sum\limits_{k=1}^n \frac{\partial f}{\partial x_k}(W(h_1),...,W(h_n)) h_k.$$
More generally, we can define the $k$th order derivative $D^kF\in L^2(\Omega, \mathfrak{H}^{\otimes k})$ by iterating the derivative operator $k$ times. For any integer $k\geq 1$ and any $p\geq 1,$ we denote by $\mathbb{D}^{k,p}$ the closure of $\mathcal{S}$ with respect to the norm
$$\|F\|^p_{k,p}:=E|F|^p+\sum\limits_{i=1}^kE\|D^i F\|^p_{\mathfrak{H}^{\otimes i}}.$$
An important operator in the Malliavin calculus theory is the divergence operator $\delta,$ it is the adjoint of the derivative operator $D$ characterized by
\begin{equation}\label{s44}
E\langle DF,u\rangle_{\mathfrak{H}}=E[F\delta(u)]
\end{equation}
for any $F\in \mathcal{S}$ and $u\in L^2(\Omega,\mathfrak{H}).$ The domain of $\delta$ is the set of all processes $u\in L^2(\Omega,\mathfrak{H})$ such that
$$E|\langle DF,u\rangle_{\mathfrak{H}}|\leq C(u)\|F\|_{L^2(\Omega)},$$
where $C(u)$ is some positive constant depending only on $u.$ Let $F\in \mathbb{D}^{1,2}$ and $u\in Dom\,\delta$ such that  $Fu\in L^2(\Omega,\mathfrak{H}).$ Then $Fu\in Dom\,\delta$ and we have the following relation
$$
\delta(Fu)=F\delta(u)-\left\langle DF,u\right\rangle _{\mathfrak{H}},
$$
provided the right-hand side is square integrable.

For any integer $q\geq 0,$ we denote by $\mathcal{H}_q$ the $q$th Wiener chaos of $W.$ We recall that $\mathcal{H}_0=\mathbb{R}$ and, for any $q\geq 1,$ $\mathcal{H}_q$ is the closed linear subspace of $L^2(\Omega)$ generated by the family of random variables $I_q(f)$ where $I_q$ indicates a multiple Wiener-It\^o integral of order $q$ and $f\in \mathfrak{H}^{\odot q}$ (the $q$th symmetric tensor power of $\mathfrak{H}$). It is known that any random variable $F$ in $L^2(\Omega)$ can be expanded into an orthogonal sum of its Wiener chaos:
$$F=E[F]+\sum\limits_{q=1}^\infty I_q(f_q),$$
where the series converges in $L^2(\Omega)$ and $f_q\in \mathfrak{H}^{\odot q}.$ From this chaos expansion one may define the Ornstein-Uhlenbeck operator $L$ by $LF=\sum\limits_{q=1}^\infty-q I_q(f_q)$ when $F\in \mathbb{D}^{2,2}$ and its pseudo-inverse by $L^{-1}(F-E[F])=\sum\limits_{q=1}^\infty-\frac{1}{q} I_q(f_q).$ Hereafter, we write $L^{-1}F$ instead of $L^{-1}(F-E[F]).$ Note that, for any $F\in L^2(\Omega),$ we have $L^{-1}F\in Dom\,L$ and $LL^{-1}F=L^{-1}LF=F-E[F].$ Moreover, the operators $D,\delta$ and $L$ satisfy the following relationship: $F\in Dom\,L$ if and only if $F\in \mathbb{D}^{2,2}$ and, in this case,
\begin{equation}\label{hj4}
\delta D F=-LF.
\end{equation}

\section{Non-uniform Berry-Esseen bounds}\label{uloq}

Let $\Phi$ denote the cumulative distribution function of standard normal random variable. We recall that, for each $z\in \mathbb{R},$ the Stein equation
\begin{equation*}
f'(x)-xf(x)=\ind_{\{x\leq z\}}-\Phi(z),\,\,x\in \mathbb{R}
\end{equation*}
admits a unique solution $f_z(x)$ given by
\begin{equation}\label{uydk}
f_z(x)=\left\{
         \begin{array}{ll}
           \sqrt{2\pi}e^{x^2/2}\Phi(x)(1-\Phi(z))& \hbox{if}\,\,\,x\leq z, \\
           \sqrt{2\pi}e^{x^2/2}(1-\Phi(x))\Phi(z) & \hbox{if}\,\,\,x> z.
         \end{array}
       \right.
\end{equation}
The key point in our proof is to establish new estimates for $f_z$ and $f'_z$ as in the next lemma.
\begin{lem}\label{erf} Let $z>0$ and let $f_z$ be given by (\ref{uydk}). Then we have
\begin{equation}\label{m9io1}
0<f_z(x)\leq \frac{\sqrt{2\pi}}{4},\,\,\,|f'_z(x)|\leq 1\,\,\forall\,x\in \mathbb{R},
\end{equation}
\begin{equation}\label{m9q}
0<f_z(x)\leq \frac{\sqrt{2\pi}}{2}e^{-z^2/4}\,\,\forall\,|x|\leq z/2
\end{equation}
and
\begin{equation}\label{m9io}
|f'_z(x)|\leq 2e^{-z^2/4}\,\,\forall\,|x|\leq z/2.
\end{equation}
\end{lem}
\begin{proof}The estimates in (\ref{m9io1}) are well known, see e.g Lemma 2.3 in \cite{Chen2011}. Let us prove (\ref{m9io}). When $x\leq z,$ we have $f_z(x)=\sqrt{2\pi}e^{x^2/2}\Phi(x)(1-\Phi(z))$ and hence,
$$f'_z(x)=(1-\Phi(z))\left(1+\sqrt{2\pi}xe^{x^2/2}\Phi(x)\right).$$
Fixed $z>0,$ we deduce
$$|f'_z(x)|\leq (1-\Phi(z))\left(1+\frac{\sqrt{2\pi}}{2}ze^{z^2/8}\right),\,\,|x|\leq z/2.$$
We now observe that
\begin{align*}
1-\Phi(z)=\int_0^\infty \frac{1}{\sqrt{2\pi}}e^{-(u+z)^2/2}du\leq e^{-z^2/2}\int_0^\infty \frac{1}{\sqrt{2\pi}}e^{-u^2/2}du=\frac{1}{2}e^{-z^2/2}.
\end{align*}
Consequently, we obtain
\begin{align*}
|f'_z(x)|&\leq \frac{1}{2}e^{-z^2/2}\left(1+\sqrt{\pi/2}ze^{z^2/8}\right)\\
&\leq \frac{1}{2}e^{-z^2/4}\left(1+\sqrt{\pi/2}ze^{-z^2/8}\right),\,\,|x|\leq z/2.
\end{align*}
Furthermore, it is easy to verify that $ze^{-z^2/8}\leq 2e^{-1/2}\,\forall\,z>0.$ So we conclude that
$$|f'_z(x)|\leq \frac{1}{2}e^{-z^2/4}\left(1+2\sqrt{\pi/2}e^{-1/2}\right)<2e^{-z^2/4},\,\,|x|\leq z/2.$$
This completes the proof of (\ref{m9io}). The verification of (\ref{m9q}) is trivial. Indeed, we have
$$f_z(x)=\sqrt{2\pi}e^{x^2/2}\Phi(x)(1-\Phi(z))\leq \frac{\sqrt{2\pi}}{2}e^{z^2/8}e^{-z^2/2}\leq \frac{\sqrt{2\pi}}{2}e^{-z^2/4},\,\,|x|\leq z/2.$$
The proof of the lemma is complete.
\end{proof}
We now are in a position to state and prove the main results of the present paper.
\begin{thm}\label{kfoi} Let $Z\sim N(0,1)$ and let $F\in \mathbb{D}^{1,2}.$ If, in addition, the law of $F$ is absolutely continuous with respect to the Lebesgue measure, then we have
\begin{equation}\label{hjo}
|P(F\leq z)-P(Z\leq z)|\leq \big(|E[F]|+\sqrt{E|1-\langle DF,-DL^{-1}F\rangle_{\mathfrak{H}}|^2}\big)\big(\sqrt{P(|F|>|z|/2)}+2e^{-z^2/4}\big)\,\,\,\forall\,z\in \mathbb{R}.
\end{equation}
\end{thm}
\begin{proof}We first recall  the arguments in the proof of Theorem 3.1 in \cite{NourdinPe2009}. Note that $F=E[F]-\delta DL^{-1}F.$ Then, for any $z\in \mathbb{R},$ by using the Stein's equation and the relationship (\ref{s44}), we have
\begin{align*}
P(F\leq z)-P(Z\leq z)&=E[f'_z(F)]-E[f_z(F)F]\\
&=E[f'_z(F)]-E[f_z(F)(F-E[F])]-E[f_z(F)]E[F]\\
&=E[f'_z(F)]+E[f_z(F)\delta DL^{-1}F]-E[f_z(F)]E[F]\\
&=E[f'_z(F)(1-\langle DF,-DL^{-1}F\rangle_{\mathfrak{H}})]-E[f_z(F)]E[F].
\end{align*}
\noindent{\bf Case 1: $z\geq 0.$}  We rewrite the above expression as follows
\begin{align*}
P&(F\leq z)-P(Z\leq z)\\
&=E[f'_z(F)(1-\langle DF,-DL^{-1}F\rangle_{\mathfrak{H}})\ind_{\{|F|\leq z/2\}}]+E[f'_z(F)(1-\langle DF,-DL^{-1}F\rangle_{\mathfrak{H}})\ind_{\{|F|> z/2\}}]\\
&-E[f_z(F)\ind_{\{|F|\leq z/2\}}]E[F]-E[f_z(F)\ind_{\{|F|> z/2\}}]E[F].
\end{align*}
Then, by Lemma \ref{erf}, we get
\begin{align*}
|P&(F\leq z)-P(Z\leq z)|\\
&\leq E|f'_z(F)(1-\langle DF,-DL^{-1}F\rangle_{\mathfrak{H}})\ind_{\{|F|\leq z/2\}}|+E|f'_z(F)(1-\langle DF,-DL^{-1}F\rangle_{\mathfrak{H}})\ind_{\{|F|> z/2\}}|\\
&+(E[f_z(F)\ind_{\{|F|\leq z/2\}}]+E[f_z(F)\ind_{\{|F|> z/2\}}])|E[F]|\\
&\leq 2e^{-z^2/4}E|1-\langle DF,-DL^{-1}F\rangle_{\mathfrak{H}}|+E\left[|1-\langle DF,-DL^{-1}F\rangle_{\mathfrak{H}}|\ind_{\{|F|> z/2\}}\right]\\
&+\bigg(\frac{\sqrt{2\pi}}{2}e^{-z^2/4}+\frac{\sqrt{2\pi}}{4}P(|F|>z/2)\bigg)|E[F]|.
\end{align*}
So, by applying the Cauchy-Schwarz inequality, we obtain
\begin{align*}
|P(F\leq z)-P(Z\leq z)|\leq \big(|E[F]|+\sqrt{E|1-\langle DF,-DL^{-1}F\rangle_{\mathfrak{H}}|^2}\big)\big(\sqrt{P(|F|>z/2)}+2e^{-z^2/4}\big).
\end{align*}
This finishes the proof of (\ref{hjo}) for $z\geq 0.$

\noindent{\bf Case 2: $z< 0.$} This case follows directly from {\bf Case 1}. Indeed, we have
\begin{align*}
|P(F\leq z)-P(Z\leq z)|&=|P(-F\geq -z)-P(-Z\geq -z)|\\
&=|P(-F\geq -z)-P(Z\geq -z)|\\
&=|P(-F\leq -z)-P(Z\leq -z)|\\
&\leq \big(|E[-F]|+\sqrt{E|1-\langle D(-F),-DL^{-1}(-F)\rangle_{\mathfrak{H}}|^2}\big)\big(\sqrt{P(|-F|>-z/2)}+2e^{-z^2/4}\big)\\
&=\big(|E[F]|+\sqrt{E|1-\langle DF,-DL^{-1}F\rangle_{\mathfrak{H}}|^2}\big)\big(\sqrt{P(|F|>-z/2)}+2e^{-z^2/4}\big).
\end{align*}
The proof of the theorem is complete.
\end{proof}
\begin{rem}
When $F$ can be represented as a Skorokhod integral: $F=E[F]+\delta(u)$ for some stochastic process $u,$ we use the relationship (\ref{s44}) to get
\begin{align*}
P(F\leq z)-P(Z\leq z)&=E[f'_z(F)]-E[f_z(F)(F-E[F])]-E[f_z(F)]E[F]\\
&=E[f'_z(F)]-E[f_z(F)\delta(u)]-E[f_z(F)]E[F]\\
&=E[f'_z(F)(1-\langle DF,u\rangle_{\mathfrak{H}})]-E[f_z(F)]E[F].
\end{align*}
Hence, with the exact proof of Theorem \ref{kfoi}, we also have the following non-uniform Berry-Esseen bound
\begin{equation}\label{h5jo}
|P(F\leq z)-P(Z\leq z)|\leq \big(|E[F]|+\sqrt{E|1-\langle DF,u\rangle_{\mathfrak{H}}|^2}\big)(\sqrt{P(|F|>|z|/2)}+2e^{-z^2/4})\,\,\,\forall\,z\in \mathbb{R}.
\end{equation}

\end{rem}
\section{Some applications}\label{9jko}
In this section, we provide some applications to illustrate the results obtained in the previous section.
\subsection{Multiple Wiener-It\^o integrals}
We recall that among the applications of the bound (\ref{whjo}) is the fourth moment theorem for the multiple Wiener-It\^o integrals. Here we have the following non-uniform Berry-Esseen bound.
\begin{thm} Let $q\geq 2$ be an integer, and let  $F= I_q(f)$ have variance one. Then, it holds that
\begin{equation}\label{hjo6}
|P(F\leq z)-P(Z\leq z)|\leq \sqrt{\frac{q-1}{3q}(E|F|^4-3)}(c_qe^{-\frac{z^{2/q}}{2^{2+2/q}}}+2e^{-\frac{z^2}{4}})\,\,\,\forall\,z\in \mathbb{R},
\end{equation}
where $c_q$ is a positive constant depending only on $q.$
\end{thm}
\begin{proof}When $F= I_q(f),$ it is known from \cite{NourdinPe2009} that
$$\sqrt{E|1-\langle DF,-DL^{-1}F\rangle_{\mathfrak{H}}|^2}\leq \sqrt{\frac{q-1}{3q}(E|F|^4-3)}.$$
On the other hand, for the multiple Wiener-It\^o integrals with variance one, Theorem 4.1 in \cite{Major2005} provides us the concentration bound
$$P\left(|I_q(f)|>x\right)\leq c_q^2\exp\left(-\frac{x^{2/q}}{2}\right),\,\,\,x\geq  0,$$
where $c_q$ is a positive constant depending only on $q.$ So the bound (\ref{hjo6}) follows directly from (\ref{hjo}).
\end{proof}
To our best knowledge, this is the first time that a non-uniform Berry-Esseen bound for the multiple Wiener-It\^o integrals have appeared in the literature. In addition, this non-uniform bound significantly improves the uniform bound obtained in \cite{NourdinPe2009}:
$$\sup\limits_{z\in \mathbb{R}}|P(F\leq z)-P(Z\leq z)|\leq \sqrt{\frac{q-1}{3q}(E|F|^4-3)}.$$

\subsection{Exponential functionals}
Let $B=(B_t)_{t\in [0, T]}$ be a standard Brownian motion defined on a complete probability space $(\Omega ,\mathcal{F},\mathbb{F},P)$, where $\mathbb{F}= (\mathcal{F}_t)_{t \in [0,T]}$ is a natural filtration generated by $B.$ We consider the exponential functional of the form
\begin{equation}\label{yu1}
F_t=\int_0^t e^{as+B_s}ds,\,\,t\in [0, T],
\end{equation}
where $a$ is a real number. It is known that this functional plays an important role in several domains.  A lot of fruitful properties of $F_t$ can be founded in the literature, see e.g. \cite{Matsumoto2005a,Matsumoto2005b,Yor2001}. Let us define
\begin{equation}\label{y6u1}
\tilde{F}_t:=\frac{F_t-m_t}{\sigma_t},\,\,t\in [0, T],
\end{equation}
where $m_t=E[F_t]$ and $\sigma_t^2={\rm Var}(F_t).$ It was proved in \cite{Dufresne2004} that $\lim\limits_{t\to 0^+}\frac{m_t}{t}=1,\,\,\,\lim\limits_{t\to 0^+}\frac{3\sigma_t^2}{t^3}=1$ and
$$\tilde{F}_t\,\,\,\text{converges in distribution to}\,\,Z\sim N(0,1)\,\,\text{as}\,\,t\to 0.$$
In the next theorem, we give an explicit estimate for the rate of convergence.
\begin{thm}\label{vnms}Let $Z\sim N(0,1)$ and let $(\tilde{F}_t)_{t\geq 0}$ be given by (\ref{y6u1}). Then, for each $t\in (0,T],$ we have
$$|P(\tilde{F}_t\leq z)-P(Z\leq z)|\leq \frac{2e^{2at+4t}t^3\sqrt{t}}{\sigma_t^2} \left(\exp\left(-\frac{\ln^2(1+\frac{|z|\sigma_t}{2m_t})}{4t}\right)+e^{-z^2/16}+2e^{-z^2/4}\right)\,\,\,\forall\,z\in \mathbb{R}.$$
\end{thm}
In order to see the rate of convergence, we note that
$$\text{$\lim\limits_{t\to 0^+}\frac{2e^{2at+4t}t^3}{\sigma_t^2}=6$ and $\lim\limits_{t\to 0^+}\frac{\ln^2(1+\frac{|z|\sigma_t}{2m_t})}{4t}=\frac{z^2}{48}$}.$$
Hereafter, the Malliavin derivative operator $D$ is with respect to $B.$ In proofs, we use the following covariance formula: For any $F,G\in \mathbb{D}^{1,2},$ we have
\begin{equation}\label{uijr}
{\rm Cov}(F,G)=E\left[\int_0^T D_sFE[D_sG|\mathcal{F}_s]ds\right].
\end{equation}
Before proving Theorem \ref{vnms}, we first provide the concentration inequalities for $\tilde{F}_t.$
\begin{prop}\label{swde} Let $(\tilde{F}_t)_{t\geq 0}$ be given by (\ref{y6u1}). Then, for each $t\in (0,T],$ we have
 \begin{equation}\label{dk1}
P(\tilde{F}_t\geq x)\leq e^{-\frac{\ln^2(1+x\sigma_t/m_t)}{2t}},\,\,\,x\geq 0
\end{equation}
and
\begin{equation}\label{dk2}
P(\tilde{F}_t\leq -x)\leq e^{-x^2/2},\,\,\,x\geq 0
\end{equation}
\end{prop}
\begin{proof}Fixed $t\in (0,T],$ we consider the random variable
$$X_t=\ln\left(\int_0^t e^{as+B_s}ds\right).$$
It is easy to verify that $X_t$ is a Malliavin differentiable random variable and its derivative  is given by
$$D_\theta X_t=\frac{\int_\theta^t e^{as+B_s}ds}{\int_0^t e^{as+B_s}ds},\,\,0\leq\theta\leq t.$$
We have $0< D_\theta X_t\leq 1$ and hence,
$$0<\int_0^t D_\theta X_tE[D_\theta X_t|\mathcal{F}_\theta]d\theta\leq t\,\,a.s.$$
This, together with Theorem 2.4 in \cite{DPT2011}, implies that
$$P\left(X_t-E[X_t]\geq y\right)\leq  e^{-\frac{y^2}{2t}}\,\,\forall\,y\geq 0.$$
We now observe that $E[X_t]\leq \ln m_t$ by Lyapunov's inequality. Then, for $x\geq 0,$ we deduce
\begin{align*}
P(\tilde{F}_t\geq x)&=P\left(X_t\geq \ln(x\sigma_t+m_t)\right)\\
&=P\left(X_t-E[X_t]\geq \ln(x\sigma_t+m_t)-E[X_t]\right)\\
&\leq P\left(X_t-E[X_t]\geq \ln(x\sigma_t+m_t)-\ln m_t\right)\\
&\leq e^{-\frac{\ln^2(1+x\sigma_t/m_t)}{2t}}.
\end{align*}
This completes the proof of (\ref{dk1}). In order to prove the relation (\ref{dk2}), we consider the function
$$g(\lambda)=E[e^{-\lambda \tilde{F}_t}],\,\,\lambda\geq 0.$$
By using the covariance formula (\ref{uijr}), we have
$$g'(\lambda)=-E[e^{-\lambda \tilde{F}_t}\tilde{F}_t]=\lambda E[e^{-\lambda \tilde{F}_t}\Gamma_t],$$
where $\Gamma_t$ is defined by
\begin{equation}\label{hhy}
\Gamma_t:=\int_0^t D_\theta \tilde{F}_tE[D_\theta \tilde{F}_t|\mathcal{F}_\theta]d\theta.
\end{equation}
Note that $E[\Gamma_t]={\rm Cov}(\tilde{F}_t,\tilde{F}_t)=1.$ Then, once again, we use (\ref{uijr}) to get
\begin{align*}
g'(\lambda)&=\lambda E[e^{-\lambda \tilde{F}_t}]+\lambda{\rm Cov}(e^{-\lambda \tilde{F}_t},\Gamma_t)\\
&=\lambda g(\lambda)-\lambda^2E\left[e^{-\lambda \tilde{F}_t}\int_0^t D_r\tilde{F}_tE[D_r\Gamma_t|\mathcal{F}_r]dr\right].
\end{align*}
We observe that $\int_0^t D_r\tilde{F}_tE[D_r\Gamma_t|\mathcal{F}_r]dr\geq 0\,\,a.s.$ Indeed, we have for $0\leq r,\theta\leq t,$
$$D_\theta \tilde{F}_t=\frac{1}{\sigma_t}\int_\theta^t e^{as+B_s}ds\geq 0,\,\,\,D_rD_\theta \tilde{F}_t=\frac{1}{\sigma_t}\int_\theta^t e^{as+B_s}\ind_{[r,t]}(s)ds\geq 0$$
and
$$D_r\Gamma_t=\int_0^t D_rD_\theta \tilde{F}_tE[D_\theta \tilde{F}_t|\mathcal{F}_\theta]d\theta+\int_r^t D_\theta \tilde{F}_tE[D_rD_\theta \tilde{F}_t|\mathcal{F}_\theta]d\theta\geq 0.$$
So it holds that
$$g'(\lambda)\leq \lambda g(\lambda),\,\,\lambda\geq 0.$$
This gives us
$$g(\lambda)\leq e^{\frac{\lambda^2}{2}},\,\,\lambda\geq 0.$$
Fixed $x\geq 0,$ by Markov’s inequality, we have
\begin{align*}
P(\tilde{F}_t\leq -x)\leq e^{-\lambda x}E[e^{-\lambda \tilde{F}_t}]=e^{-\lambda x}g(\lambda)\leq e^{-\lambda x+\frac{\lambda^2}{2}},\,\,\lambda\geq 0.
\end{align*}
Choosing $\lambda=x$ yields $P(\tilde{F}_t\leq -x)\leq e^{-x^2/2}.$ So the relation (\ref{dk2}) is verified.

The proof of the proposition is complete.
\end{proof}
\noindent{\it Proof of Theorem \ref{vnms}.}  We will carry out the proof in two steps.

\noindent{\it Step 1: Moment estimates.} By using the H\"older inequality we have
\begin{align*}
E|D_\theta \tilde{F}_t|^4&=\frac{1}{\sigma_t^4}E\bigg|\int_\theta^t e^{as+B_s}ds\bigg|^4\leq \frac{t^3}{\sigma_t^4}\int_\theta^t E[e^{4as+4B_s}]ds\\
&= \frac{t^3}{\sigma_t^4}\int_\theta^t e^{4as+8s}ds= \frac{t^4}{\sigma_t^4} e^{4at+8t}\,\,\forall\,0\leq \theta\leq t
\end{align*}
and
\begin{align*}
E|D_rD_\theta \tilde{F}_t|^4=\frac{1}{\sigma_t^4}E\bigg|\int_\theta^t e^{as+B_s}\ind_{[r,t]}(s)ds\bigg|^4\leq \frac{t^4}{\sigma_t^4} e^{4at+8t}\,\,\forall\,0\leq r,\theta\leq t.
\end{align*}
Let $\Gamma_t$ be given by (\ref{hhy}). We use the Cauchy-Schwarz and Lyapunov inequalities to deduce
\begin{align}
E|D_r\Gamma_t|^2&=E\bigg|\int_0^t D_rD_\theta \tilde{F}_tE[D_\theta \tilde{F}_t|\mathcal{F}_\theta]d\theta+\int_r^t D_\theta \tilde{F}_tE[D_rD_\theta \tilde{F}_t|\mathcal{F}_\theta]d\theta\bigg|^2\notag\\
&\leq 2t\int_0^t E|D_rD_\theta \tilde{F}_tE[D_\theta \tilde{F}_t|\mathcal{F}_\theta]|^2d\theta+2t\int_r^t E|D_\theta \tilde{F}_tE[D_rD_\theta \tilde{F}_t|\mathcal{F}_\theta]|^2d\theta\notag\\
&\leq 2t\int_0^t \sqrt{E|D_rD_\theta \tilde{F}_t|^4E|D_\theta \tilde{F}_t|^4}d\theta+2t\int_r^t \sqrt{E|D_\theta \tilde{F}_t|^4E|D_rD_\theta \tilde{F}_t|^4}d\theta\notag\\
&\leq \frac{4t^6}{\sigma_t^4} e^{4at+8t}\,\,\forall\,0\leq r\leq t.\label{9ol}
\end{align}
\noindent{\it Step 2: Conclusion.} By the Clark-Ocone formula, we have
$$\tilde{F}_t=E[\tilde{F}_t]
+\int_0^tE[\tilde{F}_t|\mathcal{F}_\theta]dB_\theta.$$
Hence, by using the bound (\ref{h5jo}) with $u_\theta=E[\tilde{F}_t|\mathcal{F}_\theta],$ we obtain
\begin{equation}
|P(\tilde{F}_t\leq z)-P(Z\leq z)|\leq \sqrt{E|1-\Gamma_t|^2}\left(\sqrt{P(|\tilde{F}_t|>|z|/2)}+2e^{-z^2/4}\right)\,\,\,\forall\,z\in \mathbb{R},
\end{equation}
Thanks to Proposition \ref{swde} we have
\begin{align*}
P(|\tilde{F}_t|>|z|/2)&=P(\tilde{F}_t>|z|/2)+P(\tilde{F}_t<-|z|/2)\\
&\leq \exp\left(-\frac{\ln^2(1+\frac{|z|\sigma_t}{2m_t})}{2t}\right)+e^{-z^2/8}\,\,\,\forall\,z\in \mathbb{R}.
\end{align*}
On the other hand, recalling (\ref{9ol}), we have
\begin{align*}
E|1-\Gamma_t|^2={\rm Cov}(\Gamma_t,\Gamma_t)\leq \int_0^t E|D_r\Gamma_t|^2dr\leq \frac{4t^7}{\sigma_t^4} e^{4at+8t}.
\end{align*}
Combining the above computations yields
$$|P(\tilde{F}_t\leq z)-P(Z\leq z)|\leq \frac{2e^{2at+4t}t^3\sqrt{t}}{\sigma_t^2} \left(\exp\left(-\frac{\ln^2(1+\frac{|z|\sigma_t}{2m_t})}{4t}\right)+e^{-z^2/16}+2e^{-z^2/4}\right)\,\,\,\forall\,z\in \mathbb{R}.$$
This finishes the proof of Theorem \ref{vnms}. \hfill $\square$


\noindent {\bf Declaration of interests:} The authors do not work for, advise, own shares in, or receive funds from any organisation that could benefit from this article, and have declared no affiliation other than their research organisations


\end{document}